\documentclass[11pt]{article}

\usepackage{fullpage}
\usepackage{amsmath,amsthm,amsfonts}
\usepackage{amssymb,latexsym,graphicx}
\usepackage{fourier}
\usepackage{mathtools}
\usepackage{hyperref, cleveref}
\usepackage{graphicx}
\usepackage{url}
\usepackage[ruled]{algorithm2e}

\makeatletter
\def\resetMathstrut@{%
  \setbox\z@\hbox{%
    \mathchardef\@tempa\mathcode`\(\relax
    \def\@tempb##1"##2##3{\the\textfont"##3\char"}%
    \expandafter\@tempb\meaning\@tempa \relax
  }%
  \ht\Mathstrutbox@1.2\ht\z@ \dp\Mathstrutbox@1.2\dp\z@
}
\makeatother

\newtheorem{theorem}{Theorem}[section]

\newtheorem{claim}[theorem]{Claim}

\newtheorem{corollary}[theorem]{Corollary}
\newtheorem{definition}[theorem]{Definition}

\newcommand{\R}{\ensuremath{\mathbb{R}}}

\newcommand{\eps}{\varepsilon}
\renewcommand{\epsilon}{\varepsilon}

\renewcommand{\leq}{\leqslant}
\renewcommand{\geq}{\geqslant}

\newcommand{\E}{\mathbb{E}}

  \newcommand{\st}{:\,} 

  \DeclareMathOperator{\sign}{sign}
  \newcommand{\poly}{\mathit{poly}}

\begin{document}
\title{The Fourier Transform of Restrictions of Functions on the Slice}
\author{Shravas Rao \thanks{Northwestern University}}
\date{\today}

\maketitle

\begin{abstract}
This paper considers the Fourier transform over the slice of the Boolean hypercube.
We prove a relationship between the Fourier coefficients of a function over the slice, and the Fourier coefficients of its restrictions.
As an application, we prove a Goldreich-Levin theorem for functions on the slice based on the Kushilevitz-Mansour algorithm for the Boolean hypercube.
\end{abstract}

\section{Introduction}

Boolean functions, or functions over $\{0, 1\}^n$, are central to computer science.
One tool that has found many uses is Fourier analysis, which involves defining an orthonormal basis separate from the standard basis, but with many useful properties.
One such property is that this basis consists of eigenvectors of the Boolean hypercube, the graph whose vertices are $\{0, 1\}^n$ with an edge between two vectors $u$ and $v$ if they differ in exactly one coordinate.
Fourier analysis over the Boolean hypercube has led to many applications, including in learning theory, property testing, pseudorandomness, voting theory, and more.
For a more complete treatment, see~\cite{O14}.

In recent years, much attention has been drawn to functions on the slice,  $\binom{[n]}{k} = \{ v \mid v \in \{0, 1\}^n, \sum_{i=1}^n v_i = k\}$ or the subset of $\{0, 1\}^n$ that consists of vectors with exactly $k$ $1$'s for some fixed $k$.
Like in the case of the Boolean hypercube, one can also define an orthonormal basis separate from the standard basis, over functions on the slice.
In particular, one can consider the eigenvectors of the Johnson graph, the graph whose vertices are elements of the slice, with an edge between two vectors $u$ and $v$ if they differ in exactly two coordinates.
Thus, this basis can be seen as an analogue of the Fourier transform, but for the slice.

Such a basis was given explicitly in~\cite{F16} and~\cite{S11}, and this basis has subsequently found many applications in analogues of major theorems for the Boolean hypercube, but now for the slice.
These include Friedgut's theorem~\cite{F16}, various versions of the invariance principle~\cite{FM19, FKMW18}, the Kindler-Safra theorem~\cite{KK20}, and a theorem due to Nisan and Szegedy regarding juntas~\cite{FI19}.
In this paper, we continue this line of work of proving analogues for the slice.

Fourier analysis on the hypercube has been particularly useful when considering restrictions of functions.
In particular, if $f: \{0, 1\}^n \rightarrow \R$ is a function, and $z \in \{0, 1\}^t$ for $t \leq n$, we let its restriction, $f_{z}: \{0, 1\}^{n-t} \rightarrow \R$, be the function defined by $f_z(x) = f(x \circ z)$ where $\circ$ denotes concatenation.
In particular, one can relate the Fourier coefficients of $f$ with the Fourier coefficients of $f_z$.

Recall that the orthonormal basis for the Fourier transform can be indexed by subsets of the set $[n]$.
Then for any $t$ and any $S \subseteq [n-t]$,
\begin{equation}\label{eq:cubestruct}
\sum_{T \in [n] \backslash [n-t]} \widehat{f}(S \cup T)^2 = \sum_{z \in \{0, 1\}^t} \widehat{f_z}(S)^2
\end{equation}
where $\widehat{f}(S)$ denotes the projection onto the vector indexed by $S$ (see Corollary 3.22 in~\cite{O14}).

In the case of the slice, one can also define a restriction, where given $f: \binom{[n]}{k} \rightarrow \R$ and $z \in \{0, 1\}^t$ for $t \leq n$, we let $f_{z}: \binom{n-t}{k-|z|} \rightarrow \R$ be defined by $f_z(x) = f(x \circ z)$, where $|z| = |\{i \st z_i = 1\}|$ denotes the Hamming weight of $z$.
If $z$ contains more than $k$ $1$'s, then $f_{z}$ is an empty function.
In this work, we show that an identity similar to Eq.~\eqref{eq:cubestruct} holds for the slice.

In the slice, the orthonormal basis for the Fourier transform is now indexed by subsets of $[n]$ with a certain property.
Let $\mathcal{H}_{n, \min\{k, n-k\}}$ be the collection of subsets $S = \{s_1, s_2, \ldots, s_i\}$ of $[n]$ of size at most $k$ such that $s_j \geq 2j$ for all $j$, when the elements of $S$ are listed in sorted order.
Then the vectors in the orthogonal basis from~\cite{F16} and~\cite{S11} can be indexed by elements of $\mathcal{H}_{n, \min\{k, n-k\}}$.
We note that a vector indexed by the set $S$ in the orthogonal basis for the slice differs significantly from the vector indexed by the same set in the orthogonal basis for the Boolean hypercube.

The main result of this paper is the following theorem, which generalizes Eq.~\eqref{eq:cubestruct} to the slice.
\begin{theorem}\label{thm:structure}
Let $f: \binom{[n]}{k} \rightarrow \R$, and let $i \leq k$.

Then for any $t$ and any $S \subseteq [n-t]$,
\[
\sum_{z \in \{0, 1\}^t} \widehat{f_z}(S)^2 = \sum_{T \subseteq [n]\backslash [n-t]} \widehat{f}(S \cup T)^2
\]
\end{theorem}
Again, $\widehat{f}(S)$ denotes the projection onto the vector indexed by $S$.
For convenience, we will let $\widehat{f}(S) = 0$ if $S$ is not a subset in $\mathcal{H}_{n, \min\{k, n-k\}}$ or $f$ is an empty function.

Eq.~\eqref{eq:cubestruct} has many applications, one of which includes the Goldreich-Levin theorem~\cite{GL89, KM93}.
Informally, this theorem says that if the Fourier coefficients of a function are concentrated on just a few vectors, then the function can be learned efficiently.
As an application of Theorem~\ref{thm:structure}, we show that a similar theorem applies in the case of the slice, following the Kushilevitz-Mansour algorithm.

\begin{theorem}\label{thm:glslice}
Given query access to a function $f \in {\binom{[n]}{k}} \rightarrow \{-1, 1\}$ and $0 < \tau \leq 1$, there is a $\poly(n, 1/\tau)$-time algorithm that with high probability outputs a list $L = \{U_1, \ldots, U_{\ell}\}$ of subsets of $[n]$ such that
\begin{itemize}
\item $|\widehat{f}(U)| \geq \tau \cdot \sqrt{\binom{n}{k}}$ implies that $U \in L$
\item $U \in L$ implies that $|\widehat{f}(U)| \geq \tau/2 \cdot \sqrt{\binom{n}{k}}$
\end{itemize}
\end{theorem}
\subsection{Future Work}

It would be interesting to see if Theorem~\ref{thm:structure} leads to other analogues of theorems for the Boolean hypercube to the slice.

It would also be interesting to see to what extent the orthonormal basis for slice can be extended to high-dimensional expanders, and additionally, if there exist analogues of theorems for the Boolean hypercube on high dimensional expanders.

Similarly to how expander graphs can be thought of as sparse approximations of the complete graph, high dimensional expanders can be thought of as sparse approximations of the Johnson graph.
In recent years, many constructions of high dimensional expanders have been developed~\cite{LSV05a, LSV05b, KO18, CTZ20}.
The theory of constructions of high dimensional expanders have already led to many applications in theoretical computer science, including mixing times of Markov chains~\cite{ALOV19, AL20, CGM21} and coding theory~\cite{AJQST20, DHKNT21}.

Unlike regular expander graphs, one has control beyond just the second eigenvalue, as shown in work by Kaufman and Oppenheim~\cite{KO20}.
In particular, the spectrum of high dimensional expanders approximates the spectrum of the Johnson graph.
Both Kaufman and Oppenheim, and work by Dikstein, Dinur, Filmus, and Harsha~\cite{DDFH18} give different decompositions into approximate eigenspaces.
The latter allows one to obtain an analogue of the Freidgut-Kalai-Naor theorem for high dimensional expanders.
It would be interesting to obtain more explicit descriptions of approximate eigenvectors in these approximate eigenspaces, that might allow one to prove more analogues for high dimensional expanders.
In particular, one hope is that the structure of the orthonormal basis for the slice can used as inspiration to construct an approximately orthonormal basis for high dimensional expanders, which can then be used to prove analogues of theorems for the Boolean hypercube, but now for high dimensional expanders.

\section{Preliminaries}

Denote by $[n]$ the set $\{1, 2, \ldots, n\}$ and by $\binom{[n]}{k}$ the set $\{ v \mid v \in \{0, 1\}^n, \sum_{i=1}^n v_i = k\}$ or the subset of $\{0, 1\}^n$ that consists of vertices with exactly $k$ $1$'s for some fixed $k$.

Given a string $x \in \{0, 1\}^n$, denote by $|x|$ the Hamming weight of $x$, that is, $|x| = \sum_{i=1}^n x_i$.

As stated in the introduction, we denote by $\mathcal{H}_{n, k}$ the set of subsets $S = \{s_1, s_2, \ldots, s_i\}$ of $[n]$ of size at most $k$ such that $s_j \geq 2j$ for all $j$, when the elements of $S$ are listed in sorted order.
Note that by Bertrand's ballot problem, there are $\binom{n}{|S|}-\binom{n}{|S|-1}$ sets in $\mathcal{H}_{n, k}$ of size $|S|$, and thus $|\mathcal{H}_{n, k}| = \binom{n}{k}$.

Given a vector $v \in \R^{N}$ we define the $\ell_{p}$-norm by
\[
\|v\|_{p}^p = \sum_{i=1}^N |v_i|^p .
\]
We define the inner product for two vectors $u, v \in \R^{N}$ to be
\[
\langle u, v\rangle = \sum_{i=1}^{N}u_i v_i.
\]
It will often be convenient to index vectors by vectors rather than integers, that is, to let $v$ be a vector in $\R^{\binom{[n]}{k}}$ for some $n$ and $k$.
It will often be convenient to identify vectors as functions, that is, to represent $v \in \R^{n}$ as $f_v: [n] \rightarrow \R$, or $v \in \R^{\binom{[n]}{k}}$ as $f_v: {\binom{[n]}{k}} \rightarrow \R$ where in both cases, one obtains $f_v$ from $v$ via $f_v(x) = v_x$.

By the Cauchy-Schwarz inequality, one can relate the $\ell_1$-norm of a vector with its $\ell_2$-norm as follows.
\begin{claim}\label{claim:l1vsl2}
For all $v \in \R^{N}$,
\[
\|v\|_{1} \leq \sqrt{N}\|v\|_{2}
\]
\end{claim}

We denote by $I_{N}$ the $\R^{N \times N}$ identity matrix.

\section{An orthonormal basis of eigenvectors for the Johnson graph}

Recall that the Johnson graph is defined by $G = (V, E)$ where $V = \binom{[n]}{k}$ and $E = \left\{(v, u) \mid \sum_{i=1}^n |v_i-u_i| = 2\right\}$.
Let $A_{n, k}$ be the corresponding adjacency matrix.

In this section we give with proof an orthonormal basis of eigenvectors for the Johnson graph.
The claims and theorems in this section will be used in Section~\ref{sec:restrict} to prove Theorem~\ref{thm:structure}.
We stress that all of the results in this section are known, and can be found in the works of Stanley~\cite{S88, S90},  Filmus~\cite{F16}, Srinivasan~\cite{S11}, and Filmus and Lindzey~\cite{FL20}.
When illustrative, however, we give the proofs of the various statements.

We start by defining a collection of matrices that will help yield the eigenvectors of $A_{n, k}$.

Let $P^{\uparrow}_{n, k} \in \R^{\binom{[n]}{k+1} \times \binom{[n]}{k}}$ be the matrix defined by
\begin{equation}
P^{\uparrow}_{n, k}(x, y) = \begin{cases} 1 & \text{if } y_i = 1 \text{ implies } x_i = 1 \\
0 & \text{otherwise}
\end{cases} \label{eq:updef}
\end{equation}
Let $P^{\downarrow}_{n, k+1} = (P^{\uparrow}_{n, k})^*$.

Let
\[
P^{\uparrow}_{n, i, k} = P^{\uparrow}_{n, k-1}P^{\uparrow}_{n, k-1}\cdots P^{\uparrow}_{n, i}
\]
and
\[
P^{\downarrow}_{n, k, i} = P^{\downarrow}_{n, i+1}P^{\downarrow}_{n, i+2}\cdots P^{\downarrow}_{n, k}.
\]
For simplicity, we let $P^{\uparrow}_{n, k, k} = P^{\downarrow}_{n, k, k} = I_{\binom{n}{k}}$.

Let $P^{\vee}_{n, k} = P^{\uparrow}_{n, k-1}P^{\downarrow}_{n, k}$ and $P^{\wedge} = P^{\downarrow}_{n, k+1}P^{\uparrow}_{n, k}$.
The following fact shows that $P^{\vee}_{n, k}$ and $P^{\wedge}_{n, k}$ differ only by a multiple of the identity matrix.
\begin{claim}\label{lem:diffuddu}
$P^{\vee}_{n, k} = P^{\wedge}_{n, k}-(n-2k)I_{\binom{[n]}{k}}$
\end{claim}

Similarly, one can show that $A_{n, k}$ differs from $P^{\vee}_{n, k}$ and $P^{\wedge}_{n, k}$ by a multiple of the identity matrix.
Thus, for the rest of this paper, we focus on determining an orthonormal basis of eigenvectors for $P^{\vee}_{n, k}$ and $P^{\wedge}_{n, k}$.

The following claim shows that to determine the eigenvectors of $P^{\vee}_{n, k}$ and $P^{\wedge}_{n, k}$, it is enough to determine the nullspace of $P^{\downarrow}_{n, i}$ for all $i \leq k$.
This will be under the assumption that $P^{\downarrow}_{n, i}$ has full rank when $i \leq n/2$, which we will prove in Claims~\ref{claim:characnull} and~\ref{claim:characortho}.

\begin{claim}\label{claim:nulltoeig}
Let $v \in \R^{\binom{[n]}{i}}$ be a vector such that $P^{\downarrow}_{n, i}v = 0$.
Then, 
\[
P^{\vee}_{n, k}P^{\uparrow}_{n, i, k}v =
(n-k-i+1)(k-i)P^{\uparrow}_{n, i, k}v.
\]
and
\[
P^{\wedge}_{n, k}P^{\uparrow}_{n, i, k}v =
(n-k-i)(k-i+1)P^{\uparrow}_{n, i, k}v.
\]
\end{claim}

The following claim shows that if $P^{\downarrow}_{n, i}v = 0$, then $P^{\uparrow}_{n, i, k}v = 0$ if $k > n-i$.
Thus, in the case that $k > n/2$, one only needs to consider the nullspace of $P^{\downarrow}_{n, i}$ for $i \leq n-k$.
In fact, we will compute $\displaystyle\frac{\left\|P^{\uparrow}_{n, |S|, k}v\right\|_2^2 }{ \left\|v\right\|_2^2}$ when $P^{\downarrow}_{n, i}v = 0$, which will be useful in Section~\ref{sec:norm}
\begin{claim}\label{claim:normofup}
Let $v \in \R^{\binom{[n]}{i}}$ be a vector such that $P^{\downarrow}_{n, i}v = 0$ for $i \leq n/2$.
Then,
\[
\frac{\left\|P^{\uparrow}_{n, i, k}v\right\|_2^2}{\|v\|_2^2}
=
\frac{(n-2i)!(k-i)!}{(n-k-i)!},
\]
if $k \leq n-i$, and is otherwise $0$.
\end{claim}
\begin{proof}
We prove this using induction on $i$.
In the base case when $k = i$, the Claim follows as $P^{\uparrow}_{n, i, k}v = v$.

Now, assume that the statement is true for $P^{\uparrow}_{n, i, k-1}v$.
Then,
\begin{align}
\left\|P^{\uparrow}_{n, i, k}v\right\|_2^2 &= \left\langle P^{\uparrow}_{n, k-1}P^{\uparrow}_{n, i, k-1}v , P^{\uparrow}_{n, k-1}P^{\uparrow}_{n, i k-1}v \right\rangle \nonumber
\\
&= \left\langle P^{\uparrow}_{n, i, k-1}v , P^{\wedge}_{n, k-1}P^{\uparrow}_{n, i, k-1}v \right\rangle \nonumber
\\
&=
(n-k-i+1)(k-i) \left\langle P^{\uparrow}_{n, i, k-1}v , P^{\uparrow}_{n, i, k-1}v \right\rangle \label{eq:innerproduct2} \\
&=
(n-k-i+1)(k-i) \frac{(n-2i)!(k-i-1)!}{(n-k-i+1)!} \nonumber \\
& 
\frac{(n-2i)!(k-i)!}{(n-k-i)!} \label{eq:innerproduct}
\end{align}
as desired, where the third equality follows from Claim~\ref{claim:nulltoeig} and the fourth equality follows from the inductive hypothesis.

If $k = n-i+1$, then Eq.~\eqref{eq:innerproduct2} is $0$, and it follows that $\left\|P^{\uparrow}_{n, i, k}v\right\|_2 = 0$ whenever $k > n-i+1$ by the definition of $P^{\uparrow}_{n, i, k}$.
\end{proof}

Note that the assumption that $P^{\downarrow}_{n, i}$ has full-rank when $i \leq n/2$ combined with Claims~\ref{claim:nulltoeig} and~\ref{claim:normofup} yields a subspace of eigenvectors with equal eigenvalue of dimension $\binom{n}{i}-\binom{n}{i-1}$.
For a fixed $k$ and summing over all $i$ from $0$ to $\min\{k, n-k\}$, this is $\binom{n}{\min\{k, n-k\}}$ eigenvectors, and thus, this gives a complete basis of eigenvectors.

The following claim shows that to determine an orthogonal basis of eigenvectors of $P^{\vee}_{n, k}$ and $P^{\wedge}_{n, k}$ it is enough to determine an orthogonal basis for the nullspace of $P^{\downarrow}_{n, i}$ for all $i \leq \min\{k, n-k\}$.
That is, applying $P^{\uparrow}_{n, i'}$ to two orthogonal vectors preserves orthogonality.
Recall that the eigenvectors of a symmetric matrix with different eigenvalues are orthogonal, and thus it is sufficient to determine an orthogonal basis for the nullspace of $P^{\downarrow}_{n, i}$ for each $i$.

\begin{claim}\label{claim:ortho}
Let $v_1$ and $v_2$ be such that $\langle v_1, v_2 \rangle = 0$ and $P^{\downarrow}_{n, i} v_i= 0$ for all $i$.
Then \[\left\langle P^{\uparrow}_{n, i, k}v_1, P^{\uparrow}_{n, i, k}v_2 \right\rangle = 0\].
\end{claim}
\begin{proof}
We prove this on induction on $k$.

When $k = i$, the statement is obvious.

Assume that $\left\langle P^{\uparrow}_{n, i, k-1}v_1, P^{\uparrow}_{n, i, k-1}v_2 \right\rangle = 0$.
Then through a similar sequence of equalities as Eq.~\eqref{eq:innerproduct}, one sees that
\begin{align*}
\left\langle P^{\uparrow}_{n, i, k}v_1, P^{\uparrow}_{n, i, k}v_2 \right\rangle
& =
(n-k-i+1)(k-i) \left\langle P^{\uparrow}_{n, i, k-1}  v_1,  P^{\uparrow}_{n, i, k-1} v_2 \right\rangle \\
&=
0
\end{align*}
as desired.
\end{proof}

We note that Claims~\ref{lem:diffuddu},~\ref{claim:nulltoeig}~\ref{claim:normofup}, and~\ref{claim:ortho} can be generalized to sequentially differential posets (of which, the collection of subsets of $[n]$ is an example), as is done in the work of Stanley~\cite{S88, S90}.



\subsection{An orthogonal basis for the nullspace of $P^{\downarrow}_{n, k}$}

In this section, we give an orthogonal basis for the nullspace of $P^{\downarrow}_{n, k}$ when $k \leq n/2$.
This basis was found by Srinivasan~\cite{S11} and independently by Filmus~\cite{F16}.
For completeness, we include the proof that this is an orthogonal basis due to Srinivasan.

Note that
\[
P^{\downarrow}_{n, k} = \begin{bmatrix}P^{\downarrow}_{n-1, k} & I_{\binom{n-1}{k}} \\
0 & P^{\downarrow}_{n-1, k-1}
\end{bmatrix}
\]
if $k \leq n/2$, where if $k = 1$, then we remove the bottom two blocks.
In particular, the fact that $P^{\downarrow}_{n, k}$ can be written recursively in terms of $P^{\downarrow}_{n-1, k}$ and $ P^{\downarrow}_{n-1, k-1}$ suggests that an orthogonal basis for the nullspace of the former matrix can be obtained from orthogonal bases for the nullspace of the latter matrices.

Let $S$ be a subset in $\mathcal{H}_{n, k}$.
Recall that this means that when the elements $S$ are listed in sorted order as $\{s_1, s_2, \ldots, s_{|S|}\}$, that $s_i \geq 2i$ for all $i$.
We associate with each $S$ a vector $\chi_{n, S}$ such that $P^{\downarrow}_{n, |S|}\chi_{n, S} = 0$.
We define $\chi_{n, S}$ recursively as follows.
\begin{equation}
\chi_{n, S} = \begin{cases} [1] & \text{if } |S| = 0 \\
\begin{bmatrix} \chi_{n-1, S}\\ 0\end{bmatrix} & \text{if } n \not\in S \\
\begin{bmatrix} -\displaystyle\frac{P^{\uparrow}_{n-1, |S|-1}\chi_{n-1, S\backslash\{n\}}}{n-2|S|+1}\\ \chi_{n-1, S\backslash\{n\}}\end{bmatrix} & \text{otherwise}
\end{cases} \label{eq:chidef}
\end{equation}
We briefly give some intuition behind restricting $S$ to only those subsets in $\mathcal{H}_{n, k}$.
Consider a set $S$ of size at most $k$ that is not in $\mathcal{H}_{n, k}$.
Let $S = \{s_1, \ldots, s_{|S|}\}$ be in sorted order, and let $i$ be the smallest integer such that $s_i < 2i$.
Then either $i = 1$ and $s_i = 1$, or $s_{i-1} \geq 2i-2$ and $s_i = 2i-1$.
Thus, $s_i-2i+1 = 0$, and $\chi_{n, S}$ is not well-defined.

By Bertrand's ballot problem, there are $\binom{n}{|S|}-\binom{n}{|S|-1}$ sets in $\mathcal{H}_{n, k}$ of size $|S|$.
Thus, proving that $\chi_{n, S}$ is in the nullspace of $P^{\downarrow}_{n, |S|}$ and that if $|S_1| = |S_2|$ and $S_1 \neq S_2$, then $\left\langle\chi_{n, S_1}, \chi_{n, S_2}\right\rangle = 0$ are both enough to obtain an orthogonal basis for the nullspace of $P^{\downarrow}_{n, k}$ and also show it has full rank.

We first show that $\chi_{n, S}$ is in fact in the nullspace of $P^{\downarrow}_{n, |S|}$.

\begin{claim}\label{claim:characnull}
For all $S \in \mathcal{H}_{n, k}$,
$P^{\downarrow}_{n, |S|}\chi_{n, S} = 0$.
\end{claim}
\begin{proof}
We prove this using induction on $n$ and $|S|$.

If $S = \{n'\}$, then 
\[
\chi_{n, S} = \begin{bmatrix}\smash[b]{\underbrace{\begin{matrix}-\frac{1}{n'-1} & -\frac{1}{n'-1} & \cdots & -\frac{1}{n'-1}\end{matrix}}_{n'-1 \text{ times}}} & 1 & 0 & 0 & \cdots \end{bmatrix}^*.
\]
Thus, $P^{\downarrow}_{n, 1}\chi_{n, S} = 0$.

If $n \not\in S$, then $P^{\downarrow}_{n, |S|}\chi_{n, S} = P^{\downarrow}_{n-1, |S|}\chi_{n, S} = 0$ by the inductive hypothesis.

If $n \in S$, then 
\begin{align*}
P^{\downarrow}_{n, |S|}\chi_{n, S} &= \begin{bmatrix}- \displaystyle\frac{P^{\wedge}_{n-1, |S|-1}\chi_{n-1, S\backslash\{n\}}}{n-2|S|+1}+\chi_{n-1, S\backslash\{n\}}\\ P^{\downarrow}_{n-1, |S|-1}\chi_{n-1, S\backslash\{n\}}\end{bmatrix} \\
&= \begin{bmatrix} -\displaystyle\frac{P^{\vee}_{n-1, |S|-1}\chi_{n-1, S\backslash\{n\}}}{n-2|S|+1}-\chi_{n-1, S\backslash\{n\}}+\chi_{n-1, S\backslash\{n\}}\\ 0\end{bmatrix}
\\
& = 0.
\end{align*}
where the second inequality follows from the relationship between $P^{\wedge}_{n-1, |S|-1}$ and $P^{\vee}_{n-1, |S|-1}$ and the inductive hypothesis, and the third inequality also follows from the inductive hypothesis.
\end{proof}

The following claim shows that if $|S_1| = |S_2|$ and $S_1 \neq S_2$, then $\chi_{n, S_1}$ and $\chi_{n, S_2}$ are orthogonal.
Recall that eigenvectors of a symmetric matrix with different eigenvalues are orthogonal.
Thus, this claim along with Claim~\ref{claim:characortho} are enough to prove that the given basis is orthogonal.

\begin{claim}\label{claim:characortho}
For $S_1, S_2 \in \mathcal{H}_{n, k}$, if $|S_1| = |S_2|$ and $S_1 \neq S_2$, then $\left\langle\chi_{n, S_1}, \chi_{n, S_2}\right\rangle = 0$.
\end{claim}
\begin{proof}
We prove this using induction on $n$.

Consider the base case of $|S_1| = |S_2| = 1$.
In particular, let $S_1 = \{n_1\}$ and let $S_2 = \{n_2\}$, and without loss of generality, assume that $n_2 > n_1$.
Then
\[
\chi_n(S_i) = \begin{bmatrix}\smash[b]{\underbrace{\begin{matrix}-\frac{1}{n_i-1} & -\frac{1}{n_i-1} & \cdots & -\frac{1}{n_i-1}\end{matrix}}_{n_i-1 \text{ times}}} & 1 & 0 & 0 & \cdots \end{bmatrix}^*.
\]
and the claim clearly holds.

For the inductive step, assume that the statement is true for all pairs $\chi_{n'}(S_1')$ and $\chi_{n'}(S_2')$ where $n' < n$ and $|S_1'| = |S_2'|$.
Then there are three cases.

\noindent\textbf{Case 1:} If neither $S_1$ nor $S_2$ contain $n$, then $\left\langle\chi_{n, S_1}, \chi_{n, S_2}\right\rangle = 0$ by the inductive hypothesis.

\noindent\textbf{Case 2:} If both $S_1$ and $S_2$ contain $n$, then 
\begin{alignat*}{2}\left\langle\chi_{n, S_1}, \chi_{n, S_2}\right\rangle &= \frac{1}{(n-2|S_2|+1)^2}&&\left\langle \chi_{n-1, S_1\backslash\{n\}}, P^{\wedge}_{n-1, |S_2|-1}\chi_{n-1, S_2\backslash\{n\}}\right\rangle
\\& &&+\left\langle\chi_{n-1, S_1\backslash\{n\}}, \chi_{n-1, S_2\backslash\{n\}}\right\rangle
\\
&= 
\frac{1}{(n-2|S_2|+1)^2}&&\left\langle\chi_{n-1, S_1\backslash\{n\}}, P^{\vee}_{n-1, |S_2|-1}\chi_{n-1, S_2\backslash\{n\}}\right\rangle
\\& &&+\frac{n-2|S_2|+2}{(n-2|S_2|+1)}\left\langle\chi_{n-1, S_1\backslash\{n\}}, \chi_{n-1, S_2\backslash\{n\}}\right\rangle
\\
&=
0
\end{alignat*}
where the second equality follows from Fact~\ref{lem:diffuddu} and the third inequality follows from the inductive hypothesis for $S_1\backslash\{n\}$ and $S_2\backslash\{n\}$, and the fact that $P^{\downarrow}_{n-1, |S_2|-1}  \chi_{n-1}(S_2\backslash\{n\}) = 0$.

\noindent\textbf{Case 3:} If only $S_2$ contains $n$, then 
\begin{align*}
\left\langle\chi_{n, S_1}, \chi_{n, S_2}\right\rangle
&=
\left\langle \chi_{n-1}(S_1), \frac{P^{\uparrow}_{n-1, |S_2|-1}\chi_{n-1, S_2\backslash\{n\}}}{n-2|S_2|+1}\right\rangle \\
&=
\left\langle P^{\downarrow}_{n-1, k}\chi_{n-1}(S_1), \frac{\chi_{n-1, S_2\backslash\{n\}}}{n-2|S_2|+1}\right\rangle \\
&=
0
\end{align*}
where the last inequality follows from the fact that $P^{\downarrow}_{n-1, |S_2|}\chi_{n-1, S_1} = 0$.
\end{proof}

We summarize below.

\begin{theorem}\label{thm:orthobasis}
The set of vectors 
\[
\left\{P^{\uparrow}_{n, |S|} \chi_{n, S} \mid S \in \mathcal{H}_{n, \min\{k, n-k\}}\right\}
\]
form an orthogonal basis of eigenvectors for $A_{n, k}$, the adjacency matrix of the Johnson graph.
\end{theorem}
\begin{proof}
The theorem follows by applying Claims~\ref{claim:nulltoeig},~\ref{claim:normofup},~\ref{claim:ortho},~\ref{claim:characnull}, and~\ref{claim:characortho}, and noticing that $|\mathcal{H}_{n, \min\{k, n-k\}}| = \binom{n}{k}$ and that eigenvectors with distinct eigenvalues are orthogonal.
\end{proof}

\subsection{A computation of the norm}\label{sec:norm}

In this section, we compute the norm of $P^{\uparrow}_{n, |S|, k}\chi_{n,S}$.
This combined with Theorem~\ref{thm:orthobasis} will yield a set of orthonormal eigenvectors that can be used to construct a Fourier transform over the slice.
The results in this section can be found in~\cite{F16}.
We give an inductive proof that follows the framework of Srinivasan~\cite{S11} and Filmus and Lindzey~\cite{FL20}.

We start by computing the norm of $\chi_{n,S}$.

\begin{claim}\label{claim:normofchi}
Let $S = \{s_1, \ldots, s_{|S|}\}$ be such that $s_1 \leq s_2 \leq \cdots \leq s_{|S|}$.
Then,
\[
\left\|\chi_{n, S}\right\|_2^2 = \prod_{i=1}^{|S|} \frac{s_i-2i+2}{s_i-2i+1}
\]
\end{claim}
\begin{proof}
We prove this using induction on $n$.  
For the base case, if $S = \{n'\}$, then 
\[
\chi_{n, S} = \begin{bmatrix}\smash[b]{\underbrace{\begin{matrix}-\frac{1}{n'-1} & -\frac{1}{n'-1} & \cdots & -\frac{1}{n'-1}\end{matrix}}_{n'-1 \text{ times}}} & 1 & 0 & 0 & \cdots \end{bmatrix}^*.
\]
and thus, 
\[
\left\|\chi_{n, S}\right\|_2^2  = 1+\frac{1}{n'-1} = \frac{n'}{n'-1}
\]
 as desired.

If $n \not\in S$, then 
\[
\left\|\chi_{n, S}\right\|_2^2 = \left\|\chi_{n-1, S}\right\|_2^2 = \prod_{i=1}^{|S|} \frac{s_i-2i+2}{s_i-2i+1}
\]
where the second inequality follows from the inductive hypothesis.

If $n \in S$, then 
\begin{align*}
\left\|\chi_{n, S}\right\|_2^2 &= \left\|\chi_{n-1, S\backslash\{n\}}\right\|_2^2+\frac{1}{(n-2|S|+1)^2}\left\|P^{\uparrow}_{n-1, |S|-1}\chi_{n-1, S\backslash\{n\}}\right\|_2^2 \\
&=\left\|\chi_{n-1, S\backslash\{n\}}\right\|_2^2+\frac{1}{(n-2|S|+1)^2}\left\langle \chi_{n-1, S\backslash\{n\}}, P^{\wedge}_{n, |S|-1}\chi_{n-1, S\backslash\{n\}}\right\rangle \\
&= \left(1+\frac{1}{n-2|S|+1}\right)\left\|\chi_{n-1, S\backslash\{n\}}\right\|_2^2 \\
&= \frac{n-2|S|+2}{n-2|S|+1}\prod_{i=1}^{k-1} \frac{s_i-2i+2}{s_i-2i+1} \\
&= \prod_{i=1}^{|S|} \frac{s_i-2i+2}{s_i-2i+1}
\end{align*}
where the third equality follows from the fact that $\chi_{n-1, S\backslash\{n\}}$ is an eigenvector of $P^{\wedge}_{n, k-1}$, and the fourth equality follows from the inductive hypothesis.
\end{proof}

Finally, we can compute $\left\|P^{\uparrow}_{n, |S|, k}\chi_{n,S}\right\|_2^2$.

\begin{claim}\label{claim:norm}
\[
\left\|P^{\uparrow}_{n, |S|, k}\chi_{n,S}\right\|_2^2 = \frac{(n-2|S|)!(k-|S|)!}{(n-k-|S|)!}\prod_{i=1}^k \frac{s_i-2i+2}{s_i-2i+1}
\]
\end{claim}
\begin{proof}
The claim follows by applying Claims~\ref{claim:normofup} and~\ref{claim:normofchi}.
\end{proof}

Finally, we give a formula definition for the Fourier coefficients of a function on the slice.
\begin{definition}
For $f \in \R^{\binom{[n]}{k}}$ and $S \in \mathcal{H}_{n, \min\{k, n-k\}}$, define
\[
\widehat{f}(S) = \frac{1}{\left\|P^{\uparrow}_{n, |S|, k}\chi_{n,S}\right\|_2}\left\langle f, P^{\uparrow}_{n, |S|, k}\chi_{n, S}\right\rangle.
\]
Here, $P^{\uparrow}$ is defined in Eq.~\eqref{eq:updef} and $\chi_{n, S}$ is defined in Eq.~\eqref{eq:chidef}.
It will be convenient to define $\widehat{f}(S)$ even when $S$ is not in $\mathcal{H}_{n, \min\{k, n-k\}}$, in which case we let $\widehat{f}(S) = 0$.
\end{definition}

\section{Restrictions of functions on the slice}\label{sec:restrict}

In this section, we state and prove our main theorem regarding a relationship between the Fourier coefficients of a function on the slice, and the Fourier coefficients of its restrictions.

\begin{definition}[Restriction]
Let $f: \binom{[n]}{k} \rightarrow \R$.
For $m < n$, and $z \in \{0, 1\}^m$, let $f_{z}: {\binom{[n-m]}{k-|z|}} \rightarrow \R$ (where $|z| = \sum_{i=1}^m z_i$ is the Hamming weight of $z$) be the function defined by $f_z(x) = f(x \circ z)$ where $\circ$ denotes concatenation.
\end{definition}

We start with the following theorem.

\begin{theorem}\label{thm:restrict}
Let $f: \binom{[n]}{k} \rightarrow \R$.
Then, for all $S \in \mathcal{H}_{n-1, \min\{k, n-k\}}$,
\[
\widehat{f}(S) = \frac{1}{\sqrt{n-2|S|}}\left(\sqrt{n-k-|S|}\widehat{f_0}(S)+\sqrt{k-|S|}\widehat{f_1}(S)\right).
\]
Additionally, if $|S| < \min\{k, n-k\}$, then
\[
\widehat{f}(S\cup \{n\}) = \frac{1}{\sqrt{n-2|S|}}\left(-\sqrt{k-|S|}\widehat{f_0}(S)+\sqrt{n-k-|S|}\widehat{f_1}(S)\right)
\]

Thus, if $|S|< \min\{k, n-k\}$, then
\[
\widehat{f_0}(S) = \frac{1}{\sqrt{n-2|S|}}\left(\sqrt{n-k-|S|}\widehat{f}(S)-\sqrt{k-|S|}\widehat{f}(S\cup \{n\})\right)
\]
and
\[
\widehat{f_1}(S) =  \frac{1}{\sqrt{n-2|S|}}\left(\sqrt{n-k-|S|}\widehat{f}(S\cup \{n\})+\sqrt{k-|S|}\widehat{f}(S)\right)
\]
\end{theorem}

\begin{proof}
If $k = 0$, then $f = f_0$, and $|S|=0$, and the theorem holds.
Similarly, if $k = n$, then $f = f_1$ and $|S|=0$ and the theorem also holds.

If $|S| = k$, then $\widehat{f}(S) = \widehat{f_0}(S)$ as $\chi_{n, S}(x) = 0$ if $x_n = 1$, and again, the theorem holds.
If $|S| = n-k$, then recall that $P^{\uparrow}_{n-1, |S|, k}\chi_{n-1,S} = 0$ by Claim~\ref{claim:normofup}, and thus $\widehat{f}(S) = \widehat{f_1}(S)$, and again, the theorem holds.

Thus, we can assume that $f_0$ and $f_1$ are not empty functions, and that $S \in \mathcal{H}_{n-1, \min\{k-1, n-k-1\}}$, and in particular, has a corresponding vector in the orthogonal basis for both $A_{n-1, k}$ and $A_{n-1, k-1}$.

Let $f_0': {\binom{[n]}{k}} \rightarrow \R$ be defined by $f_0'(x) = f(x)$ if $x_n = 0$ and $f_0'(x) = 0$ otherwise.
Additionally, let $f_1': {\binom{[n]}{k}} \rightarrow \R$ be defined by $f_1'(x) = f(x)$ if $x_n = 1$, and $f_1'(x) = 0$ otherwise.

Let $i = |S|$.
Note that because $f = f'_0+f'_1$, for all $S$,
\begin{align}
\widehat{f}(S) &= \frac{1}{\left\|P^{\uparrow}_{n, i, k}\chi_{n,S}\right\|_2}\left\langle f_0', P^{\uparrow}_{n, i, k}\chi_{n,S}\right\rangle+\frac{1}{\left\|P^{\uparrow}_{n, i, k}\chi_{n,S}\right\|_2}\left\langle f_1', P^{\uparrow}_{n, i, k}\chi_{n, S}\right\rangle \label{eq:split}
\end{align}
We start with the first inequality.
We have that $\left\langle f_0', P^{\uparrow}_{n, i, k}\chi_{n,S}\right\rangle = \left\langle f_0, P^{\uparrow}_{n-1, i, k}\chi_{n-1, S}\right\rangle$, as $f_0'(x) = 0$ if $x_n = 0$.
Thus, the right-hand side of Eq.~\eqref{eq:split} is equal to
\begin{align}\
\frac{1}{\left\|P^{\uparrow}_{n, i, k}\chi_{n, S}\right\|_2}\left\langle f_0, P^{\uparrow}_{n-1, i, k}\chi_{n-1, S}\right\rangle&+\frac{1}{\left\|P^{\uparrow}_{n, i, k}\chi_{n, S}\right\|_2}\left\langle P^{\downarrow}_{n, k, i} f_1', \chi_{n, S}\right\rangle \nonumber \\
&=
\frac{\left\|P^{\uparrow}_{n-1, i, k}\chi_{n-1, S}\right\|_2\widehat{f_0}(S)}{\left\|P^{\uparrow}_{n, i, k}\chi_{n, S}\right\|_2}+\frac{k-i}{\left\|P^{\uparrow}_{n, i, k}\chi_{n, S}\right\|_2}\left\langle P^{\downarrow}_{n-1, k-1, i} f_1, \chi_{n-1, S}\right\rangle \nonumber \\
&=
\frac{\left\|P^{\uparrow}_{n-1, i, k}\chi_{n-1, S}\right\|_2\widehat{f_0}(S)}{\left\|P^{\uparrow}_{n, i, k}\chi_{n, S}\right\|_2}+\frac{k-i}{\left\|P^{\uparrow}_{n, i, k}\chi_{n, S}\right\|_2}\left\langle f_1,  P^{\uparrow}_{n-1, i, k-1} \chi_{n-1, S}\right\rangle \nonumber \\
&=
\frac{\left\|P^{\uparrow}_{n-1, i, k}\chi_{n-1, S}\right\|_2\widehat{f_0}(S)}{\left\|P^{\uparrow}_{n, i, k}\chi_{n, S}\right\|_2}+\frac{(k-i)\widehat{f_1}(S)\left\| P^{\uparrow}_{n-1, i, k-1} \chi_{n-1, S}\right\|_2}{\left\|P^{\uparrow}_{n, i, k}\chi_{n, S}\right\|_2} \label{eq:notin}
\end{align}
where the sequence of inequalities uses the fact that $\chi_{n, S}(x) = \chi_{n-1, S}(x)$ unless $x_n = 1$, in which case $\chi_{n, S}(x) = 0$.
The first equality also uses the fact that $\left(P^{\downarrow}_{n, k, i} f_1'\right)(x) = (k-i)\left(P^{\downarrow}_{n-1, k-1, i} f_1\right)(x)$ if $x_n = 0$.
By Claim~\ref{claim:norm}, Eq.~\eqref{eq:notin} is equal to
\begin{align*}
 \sqrt{\frac{n-k-i}{n-2i}}\widehat{f_0}(S)+\frac{(k-i)}{\sqrt{(n-2i)(k-i)}}\widehat{f_1}(S) = \frac{1}{\sqrt{n-2i}}\left(\sqrt{n-k-i}\widehat{f_0}(S)+\sqrt{k-i}\widehat{f_1}(S)\right)
\end{align*}
as desired.

Now, let $i = |S|+1$, and let $S' = S \cup \{n\}$.
To prove the second inequality, note that the right-hand side of Eq.~\eqref{eq:split} is equal to
\begin{align}\frac{1}{\left\|P^{\uparrow}_{n, i, k}\chi_{n, S'}\right\|_2}&\left\langle P^{\downarrow}_{n, k, i} f_0', \chi_{n, S'}\right\rangle+\frac{1}{\left\|P^{\uparrow}_{n, i, k}\chi_{n, S'}\right\|_2}\left\langle P^{\downarrow}_{n, k, i} f_1', \chi_{n, S'}\right\rangle \nonumber \\
&= \frac{-\left\langle P^{\downarrow}_{n-1, k, i} f_0, P^{\uparrow}_{n-1, i-1}\chi_{n-1, S'\backslash \{n\}}\right\rangle}{(n-2i+1)\left\|P^{\uparrow}_{n, i, k}\chi_{n, S'}\right\|_2} \nonumber \\
&+\frac{\frac{-(k-i)}{n-2i+1}\left\langle P^{\downarrow}_{n-1, k-1, i} f_1, P^{\uparrow}_{n-1, i-1}\chi_{n-1, S'\backslash \{n\}}\right\rangle+\left\langle P^{\downarrow}_{n-1, k-1, i-1} f_1, \chi_{n-1, S'\backslash \{n\}}\right\rangle}{\left\|P^{\uparrow}_{n, i, k}\chi_{n, S'}\right\|_2} \nonumber \\
&= \frac{-\left\langle P^{\downarrow}_{n-1, k, i-1} f_0,\chi_{n-1, S'\backslash \{n\}}\right\rangle}{(n-2i+1)\left\|P^{\uparrow}_{n, i, k}\chi_{n, S'}\right\|_2}
+\frac{n-k-1+1}{n-2i+1}\left\langle P^{\downarrow}_{n-1, k-1, i-1} f_1, \chi_{n-1, S'\backslash \{n\}}\right\rangle
\nonumber \\
&= \frac{-\left\|P^{\uparrow}_{n-1, i-1, k}\chi_{n, S'\backslash \{n\}}\right\|_2}{(n-2i+1)\left\|P^{\uparrow}_{n, i, k}\chi_{n, S'}\right\|_2}\widehat{f_0}(S'\backslash \{n\})+\frac{(n-k-i+1)\left\|P^{\uparrow}_{n-1, i-1, k-1}\chi_{n, S'\backslash \{n\}}\right\|_2}{(n-2i+1)\left\|P^{\uparrow}_{n, i, k}\chi_{n, S'}\right\|_2}\widehat{f_1}(S'\backslash \{n\})
 \label{eq:in}
\end{align}
where the first equality follows from the definition of $\chi_{n, S}$ along with the fact that $\left(P^{\downarrow}_{n, k, i} f_1'\right)(x) = (k-i)\left(P^{\downarrow}_{n-1, k-1, i} f_1\right)(x)$ if $x_n = 0$.
Again, by Claim~\ref{claim:norm}, Eq~\eqref{eq:in} is equal to
\begin{multline*}
\frac{-1}{n-2i+1}\sqrt{\frac{(n-2i+1)(n-2i+1)(k-i+1)}{n-2i+2}}\widehat{f_0}(S'\backslash \{n\}) + \frac{n-k-i+1}{n-2i+1}\sqrt{\frac{(n-2i+1)(n-2i+1)}{(n-2i+2)(n-k-i+1)}}\widehat{f_1}(S'\backslash \{n\}) \\
= \frac{1}{\sqrt{n-2i+2}}\left(-\sqrt{k-i+1}\widehat{f_0}(S'\backslash \{n\})+\sqrt{n-k-i+1}\widehat{f_1}(S'\backslash \{n\})\right)
\end{multline*}
as desired.
\end{proof}

Finally, we prove the following relationship between the Fourier coefficients of $f$ and the Fourier coefficients of restrictions of $f$.

\begin{corollary}\label{cor:structure}
Let $f: \binom{[n]}{k} \rightarrow \R$, and let $i \leq k$.
Then for all $S \in \mathcal{H}_{\min\{k, n-k\}}$,
\[
\widehat{f_0}(S)^2+\widehat{f_1}(S)^2 = \widehat{f}(S)^2+\widehat{f}(S\cup \{n\})^2.
\]
Additionally,
\[
\sum_{z \in \{0, 1\}^t} \widehat{f_z}(S)^2 = \sum_{T \subseteq [n]\backslash [n-t]} \widehat{f}(S \cup T)^2
\]
\end{corollary}
\begin{proof}
The first statement follows from Theorem~\ref{thm:restrict}, and the second follows by induction.
\end{proof}

\section{A Goldreich-Levin Theorem for the slice}

In this section, we prove Theorem~\ref{thm:glslice}, a Goldreich-Levin theorem for the slice.
Informally, this theorem states that the large Fourier coefficients of a function can be efficiently detected.
With the machinery of Corollary~\ref{cor:structure}, the algorithm and proof of correctness are essentially the same as for the Boolean hypercube.
However, some care is required in estimating quantities of the form
\[
\sum_{S \subseteq [n] \backslash [i]} \widehat{f}(U \cup S)^2
\]
for a given function $f$ and a given set $U \subseteq [i]$.
For completeness, we give the algorithm.

The algorithm uses a divide and conquer strategy due to Kushilevitz and Mansour~\cite{KM93}.

\begin{algorithm}[H]
\SetAlgoLined
 Let $L \gets \{\emptyset\}$

 \For{$i\gets 0$ \KwTo $n-1$}{
    \For{Each $U \in L$}{
	Estimate $\sum_{S \subseteq [n] \backslash [i]} \widehat{f}(U \cup S)^2$ by $w_1$

	Estimate $\sum_{S \subseteq [n] \backslash [i]} \widehat{f}(U \cup \{i+1\} \cup S)^2$ by $w_2$

	Remove $U$ from $L$.

	\If{$w_1 \geq \tau^2/2$}{
		Add $U$ to $L$
	}
	\If{$w_2 \geq \tau^2/2$}{
		Add $U \cup \{i+1\}$ to $L$
	}
	}
    }
 \caption{Goldreich-Levin for the slice}
\end{algorithm}

Much of the analysis is identical to the algorithm for the Boolean hypercube.
As in the case of the Boolean hypercube, the algorithm identifies a set of Fourier coefficients with large total weight, starting with the set of all Fourier coefficients.
At each step, it splits each set into two, estimates the total weights of all of the subsets, and removes those subsets with small total weight.
If the algorithm is able to estimate the total weight accurately, then at any given point, it will only identify $\poly(\tau)$ different sets.
(For a complete analysis, one can also refer to Section 3.5 of~\cite{O14}).

Thus, it is enough to show that one can estimate $\sum_{S \subseteq [n] \backslash [i]} \widehat{f}(U \cup S)^2$ and $\sum_{S \subseteq [n] \backslash [i]} \widehat{f}(U \cup \{i\} \cup S)^2$ in $\poly(n, 1/\tau)$-time to error $(\tau^2/2) \binom{n}{k}$ with probability at least $1-1/\poly(n, 1/\tau)$.

\begin{theorem}
Given query access to a function $f: {\binom{[n]}{k}} \rightarrow \{-1, 1\}$ and a set $U \subseteq [i]$ for some integer $i$, then $\poly(n, 1/\eps, 1/\tau)$ samples are sufficient to compute $\sum_{S \in [n] \backslash [i]} \widehat{f}(U \cup S)^2$ up to error $\eps \binom{[n]}{k}$ with probability at least $1-1/\poly(n, 1/\tau)$.
\end{theorem}
\begin{proof}
By Corollary~\ref{cor:structure}, we have
\begin{equation}
\frac{1}{\binom{n}{k}}  \sum_{S \subseteq [n]\backslash [i]} \widehat{f}(U \cup S)^2 = \sum_{z \in \{0, 1\}^{n-i}} \frac{\binom{i}{k-|z|}}{\binom{n}{k}} \frac{1}{\binom{i}{k-|z|}}\widehat{f_z}(U)^2. \label{eq:restrict}
\end{equation}
Let $x$ be a random variable over $\{0, 1\}^{n-i}$ where $x = z$ with probability $\frac{\binom{i}{k-|z|}}{\binom{n}{k}}$.
Note that such a probability distribution is well-defined, as $\sum_{z \in \{0, 1\}^{n-i}} \binom{i}{k-|z|} = \binom{n}{k}$.
Then the right-hand side of Eq.~\eqref{eq:restrict} can be rewritten as
\[
\E_{x}\left[\frac{\widehat{f_x}(S)^2}{\binom{i}{k-|x|}}\right] = \E_{x}\left[\frac{1}{\left\|P^{\uparrow}_{i, |S|, k-|x|}\chi_{i, S}\right\|_2^2}\frac{\langle f_x, P^{\uparrow}_{i, |S|, k-|x|}\chi_{i, S} \rangle^2}{\binom{i}{k-|x|}}\right].
\]
For each $x$ in $\{0, 1\}^{n-i}$, let $y_1$ and $y_2$ be independent random variables over $\binom{[i]}{k-|x|}$ distributed according to the absolute value of the coordinates $P^{\uparrow}_{i, |S|, k-|x|}\chi_{i, S}$.
In particular, we let the probability that $y_1 = T$ be $\frac{\left|P^{\uparrow}_{i, |S|, k-|x|}\chi_{i, S}\right|}{\left\|P^{\uparrow}_{i, |S|, k-|x|}\chi_{i, S}\right\|_1}$
and thus we obtain
\[
\E_{x}\left[\frac{\widehat{f_x}(S)^2}{\binom{i}{k-|x|}}\right] = \E_{x, y_1, y_2}\left[\frac{\left\|P^{\uparrow}_{i, |S|, k-|x|}\chi_{i, S}\right\|_1^2}{\left\|P^{\uparrow}_{i, |S|, k-|x|}\chi_{i, S}\right\|_2^2}\frac{f_z(y_1)f_z(y_2)\sign\left(\left(P^{\uparrow}_{i, |S|, k-|x|}\chi_{i, S}\right)_{y_1}\right)\sign\left(\left(P^{\uparrow}_{i, |S|, k-|x|}\chi_{i, S}\right)_{y_2}\right)}{\binom{i}{k-|x|}}\right]
\]
By Claim~\ref{claim:l1vsl2}, and the fact that $f(S) \in \{-1, 1\}$ for all $S$, it follows that the random variable inside the expectation is bounded above by $1$.
Thus, by a Chernoff bound, $O(\log(n)/\eps^2)$ samples are sufficient to compute $\sum_{S \subseteq [n] \backslash [i]} \widehat{f}(U \cup S)^2$ up to error $\eps \binom{n}{k}$ with probability at least $1-1/\poly(n, 1/\tau)$.
\end{proof}

\bibliographystyle{alphaabbrv}
\bibliography{slice}

\newcommand{\etalchar}[1]{$^{#1}$}
\begin{thebibliography}{FKMW18}
\expandafter\ifx\csname urlstyle\endcsname\relax
  \providecommand{\doi}[1]{doi:\discretionary{}{}{}#1}\else
  \providecommand{\doi}{doi:\discretionary{}{}{}\begingroup
  \urlstyle{rm}\Url}\fi

\bibitem[AJQ{\etalchar{+}}20]{AJQST20}
V.~L. Alev, F.~G. Jeronimo, D.~Quintana, S.~Srivastava, and M.~Tulsiani.
\newblock List decoding of direct sum codes.
\newblock In \emph{Proceedings of the 2020 {ACM}-{SIAM} {S}ymposium on
  {D}iscrete {A}lgorithms}, pages 1412--1425. SIAM, Philadelphia, PA, 2020.

\bibitem[AL20]{AL20}
V.~L. Alev and L.~C. Lau.
\newblock Improved analysis of higher order random walks and applications.
\newblock In \emph{S{TOC} '20---{P}roceedings of the 52nd {A}nnual {ACM}
  {SIGACT} {S}ymposium on {T}heory of {C}omputing}, pages 1198--1211. ACM, New
  York, [2020] \copyright 2020.
\newblock \doi{10.1145/3357713.3384317}.

\bibitem[ALGV19]{ALOV19}
N.~Anari, K.~Liu, S.~O. Gharan, and C.~Vinzant.
\newblock Log-concave polynomials {II}: {H}igh-dimensional walks and an {FPRAS}
  for counting bases of a matroid.
\newblock In \emph{S{TOC}'19---{P}roceedings of the 51st {A}nnual {ACM}
  {SIGACT} {S}ymposium on {T}heory of {C}omputing}, pages 1--12. ACM, New York,
  2019.
\newblock \doi{10.1145/3313276.3316385}.

\bibitem[CGM21]{CGM21}
M.~Cryan, H.~Guo, and G.~Mousa.
\newblock Modified log-{S}obolev inequalities for strongly log-concave
  distributions.
\newblock \emph{Ann. Probab.}, 49(1):506--525, 2021.
\newblock ISSN 0091-1798.
\newblock \doi{10.1214/20-AOP1453}.

\bibitem[CTZ20]{CTZ20}
D.~Conlon, J.~Tidor, and Y.~Zhao.
\newblock Hypergraph expanders of all uniformities from {C}ayley graphs.
\newblock \emph{Proc. Lond. Math. Soc. (3)}, 121(5):1311--1336, 2020.
\newblock ISSN 0024-6115.
\newblock \doi{10.1112/plms.12371}.

\bibitem[DDFH18]{DDFH18}
Y.~Dikstein, I.~Dinur, Y.~Filmus, and P.~Harsha.
\newblock Boolean function analysis on high-dimensional expanders.
\newblock In \emph{Approximation, randomization, and combinatorial
  optimization. {A}lgorithms and techniques}, volume 116 of \emph{LIPIcs.
  Leibniz Int. Proc. Inform.}, pages Art. No. 38, 20. Schloss Dagstuhl.
  Leibniz-Zent. Inform., Wadern, 2018.

\bibitem[DHK{\etalchar{+}}21]{DHKNT21}
I.~Dinur, P.~Harsha, T.~Kaufman, I.~L. Navon, and A.~Ta-Shma.
\newblock List-decoding with double samplers.
\newblock \emph{SIAM J. Comput.}, 50(2):301--349, 2021.
\newblock ISSN 0097-5397.
\newblock \doi{10.1137/19M1276650}.

\bibitem[FI19]{FI19}
Y.~Filmus and F.~Ihringer.
\newblock Boolean constant degree functions on the slice are juntas.
\newblock \emph{Discrete Math.}, 342(12):111614, 7, 2019.
\newblock ISSN 0012-365X.
\newblock \doi{10.1016/j.disc.2019.111614}.

\bibitem[Fil16]{F16}
Y.~Filmus.
\newblock An orthogonal basis for functions over a slice of the {B}oolean
  hypercube.
\newblock \emph{Electron. J. Combin.}, 23(1):Paper 1.23, 27, 2016.

\bibitem[FKMW18]{FKMW18}
Y.~Filmus, G.~Kindler, E.~Mossel, and K.~Wimmer.
\newblock Invariance principle on the slice.
\newblock \emph{ACM Trans. Comput. Theory}, 10(3):Art. 11, 37, 2018.
\newblock ISSN 1942-3454.
\newblock \doi{10.1145/3186590}.

\bibitem[FL20]{FL20}
Y.~Filmus and N.~Lindzey.
\newblock Murali’s basis for the slice, 2020.

\bibitem[FM19]{FM19}
Y.~Filmus and E.~Mossel.
\newblock Harmonicity and invariance on slices of the {B}oolean cube.
\newblock \emph{Probab. Theory Related Fields}, 175(3-4):721--782, 2019.
\newblock ISSN 0178-8051.
\newblock \doi{10.1007/s00440-019-00900-w}.

\bibitem[GL89]{GL89}
O.~Goldreich and L.~A. Levin.
\newblock A hard-core predicate for all one-way functions.
\newblock In D.~S. Johnson, editor, \emph{Proceedings of the 21st Annual {ACM}
  Symposium on Theory of Computing, May 14-17, 1989, Seattle, Washigton,
  {USA}}, pages 25--32. {ACM}, 1989.
\newblock \doi{10.1145/73007.73010}.

\bibitem[KK20]{KK20}
N.~Keller and O.~Klein.
\newblock A structure theorem for almost low-degree functions on the slice.
\newblock \emph{Israel J. Math.}, 240(1):179--221, 2020.
\newblock ISSN 0021-2172.
\newblock \doi{10.1007/s11856-020-2062-4}.

\bibitem[KM93]{KM93}
E.~Kushilevitz and Y.~Mansour.
\newblock Learning decision trees using the {F}ourier spectrum.
\newblock \emph{SIAM J. Comput.}, 22(6):1331--1348, 1993.
\newblock ISSN 0097-5397.
\newblock \doi{10.1137/0222080}.

\bibitem[KO18]{KO18}
T.~Kaufman and I.~Oppenheim.
\newblock Construction of new local spectral high dimensional expanders.
\newblock In \emph{S{TOC}'18---{P}roceedings of the 50th {A}nnual {ACM}
  {SIGACT} {S}ymposium on {T}heory of {C}omputing}, pages 773--786. ACM, New
  York, 2018.
\newblock \doi{10.1145/3188745.3188782}.

\bibitem[KO20]{KO20}
T.~Kaufman and I.~Oppenheim.
\newblock High order random walks: beyond spectral gap.
\newblock \emph{Combinatorica}, 40(2):245--281, 2020.
\newblock ISSN 0209-9683.
\newblock \doi{10.1007/s00493-019-3847-0}.

\bibitem[LSV05a]{LSV05a}
A.~Lubotzky, B.~Samuels, and U.~Vishne.
\newblock Explicit constructions of {R}amanujan complexes of type {$\~A_d$}.
\newblock \emph{European J. Combin.}, 26(6):965--993, 2005.
\newblock ISSN 0195-6698.
\newblock \doi{10.1016/j.ejc.2004.06.007}.

\bibitem[LSV05b]{LSV05b}
A.~Lubotzky, B.~Samuels, and U.~Vishne.
\newblock Ramanujan complexes of type {$\~A_d$}.
\newblock volume 149, pages 267--299. 2005.
\newblock \doi{10.1007/BF02772543}.
\newblock Probability in mathematics.

\bibitem[O'D14]{O14}
R.~O'Donnell.
\newblock \emph{Analysis of {B}oolean functions}.
\newblock Cambridge University Press, New York, 2014.
\newblock ISBN 978-1-107-03832-5.
\newblock \doi{10.1017/CBO9781139814782}.

\bibitem[Sri11]{S11}
M.~K. Srinivasan.
\newblock Symmetric chains, {G}elfand-{T}setlin chains, and the {T}erwilliger
  algebra of the binary {H}amming scheme.
\newblock \emph{J. Algebraic Combin.}, 34(2):301--322, 2011.
\newblock ISSN 0925-9899.
\newblock \doi{10.1007/s10801-010-0272-2}.

\bibitem[Sta88]{S88}
R.~P. Stanley.
\newblock Differential posets.
\newblock \emph{J. Amer. Math. Soc.}, 1(4):919--961, 1988.
\newblock ISSN 0894-0347.
\newblock \doi{10.2307/1990995}.

\bibitem[Sta90]{S90}
R.~P. Stanley.
\newblock Variations on differential posets.
\newblock In \emph{Invariant theory and tableaux ({M}inneapolis, {MN}, 1988)},
  volume~19 of \emph{IMA Vol. Math. Appl.}, pages 145--165. Springer, New York,
  1990.

\end{thebibliography}

\end{document}